\documentclass[12pt]{amsart}

\usepackage[margin=1in]{geometry}
\usepackage{amsmath,amssymb,amsthm,mathtools}
\usepackage[T1]{fontenc}
\usepackage{lmodern}
\usepackage{microtype}
\usepackage{enumitem}
\usepackage{graphicx}
\usepackage{booktabs}
\usepackage{longtable}
\usepackage{array}
\usepackage{tikz}
\usepackage{float}
\usepackage{hyperref, comment}
\allowdisplaybreaks
\usetikzlibrary{arrows.meta,calc}

\hypersetup{
  colorlinks=true,
  linkcolor=blue,
  urlcolor=blue,
  citecolor=blue
}

\newtheorem{theorem}{Theorem}
\newtheorem{lemma}{Lemma}
\newtheorem{remark}{Remark}
\newtheorem{definition}{Definition}

\newcommand{\Nzero}{\mathbb N_0}

\newcommand{\DD}{\mathbb D}
\newcommand{\CC}{\mathbb C}
\newcommand{\HH}{\mathcal H}
\newcommand{\RR}{\mathbb R}

\title{The normalized orbit of a bounded normal operator can be a frame}

\author{Ilya A. Krishtal, G\"otz E. Pfander}
\address{Department of Mathematical Sciences, Northern Illinois University, DeKalb, IL, USA\\
email: ikrishtal@niu.edu}
\address{Lehrstuhl Wissenschaftliches Rechnen, Mathematisch-Geographische Fakult\"at, Katholische Universit\"at Eichst\"att Ingolstadt, Germany\\
email: pfander@ku.de}
\date{June 18, 2026}

\begin{document}

\begin{abstract}
Conjecture 3 in  \cite{AldroubiCabrelliKrishtalMolter2026} postulates that for any bounded normal operator $T$ on a Hilbert
space $H$ and any vector $g\in H$ the system
\[
  \left\{\frac{T^k g}{\|T^k g\|}: k=0,1,2,\ldots\right\}
\]
is not a frame. It was motivated by \cite{AldroubiCabrelliCakmakMolterPetrosyan2017}, where it was established that such frames do not exist when $T$ is a self adjoint operator.  We show, however, that this conjecture is false by presenting a construction of $H$, $T$, and $g$ such that the normalized orbit considered  is indeed a frame. The operator is diagonal and is defined via a decomposition of the space into finite
blocks rapidly increasing in size. We also provide an $\epsilon$-perturbation $S$ of the operator $T$ such that the system 
\[
  \left\{{S^k g}: k=0,1,2,\ldots\right\}
\]
is a Carleson frame in the sense of \cite{AldroubiCabrelliMolterTang2017,  ChristensenHasannasabPhilippStoeva2024Mystery}.
The constructions were achieved using ChatGPT, whose assistance  was also  employed in the preparation of this manuscript.

\end{abstract}

\maketitle

\section{Introduction and results}

Frames generated by operator orbits are a central model in dynamical sampling; see
\cite{AldroubiCabrelliMolterTang2017,AldroubiCabrelliKrishtalMolter2026} for the
basic dynamical-sampling viewpoint and background survey. As such, they have generated considerable interest and appear in a large body of recent research (see \cite[and references therein]{AldroubiCabrelliKrishtalMolter2026}). Despite their
 rigidity, normal-operator orbits are of particular interest, as they can be analyzed with a richer variety of tools and are potentially better implementable in applications \cite{MST26}. The problem solved in this manuscript originates in the work of  Aldroubi, Cabrelli, \c{C}akmak, Molter, and Petrosyan \cite{AldroubiCabrelliCakmakMolterPetrosyan2017}, which is a prime example of an analysis of
iterative actions of normal operators. We also note the work of 
Philipp, who studied Bessel orbits of normal operators in \cite{Philipp2017}, as well as the 
finite-generator and finite-index research which appears in
\cite{CabrelliMolterPaternostroPhilipp2020,CabrelliMolterSuarez2021}.

The normalized-orbit question is closely related to the broader problem of
rescaling a sequence of vectors so that it becomes a frame.  For this point of view see
Yu's work on frame-normalizable sequences \cite{Yu2024FrameNormalizable}.  A
recent preprint of Yu \cite{Yu2025Spectrum} studies spectral restrictions for
normal operators whose iterative systems can be rescaled to frames.  The present
example is diagonal and has pure point spectrum inside the open unit disk, with
accumulation on the unit circle.  Thus it is designed to avoid the continuous
spectrum obstructions emphasized in \cite{Yu2025Spectrum}.

For the unnormalized diagonal-contraction version of the problem, the standard
Hardy-space model brings in Carleson interpolation and normalized reproducing
kernels.  The classical references are Carleson's interpolation theorem
\cite{Carleson1958} and the Shapiro--Shields characterization of interpolating
sequences for Hardy spaces \cite{ShapiroShields1961}; useful book references
include Garnett \cite{Garnett2007}, Nikolski \cite{Nikolski2002}, and Seip
\cite{Seip2004}.  For the reproducing-kernel and sampling/interpolation point
of view, see also
\cite{AlemanHartzMcCarthyRichter2019, BaranovDyakonov2011, FuhrGrochenigHaimiKlotzRomero2017,GroechenigHaimiOrtegaCerdaRomero2019, LataPaulsen2011}.
The terminology and recent operator-orbit use of ``Carleson frames'' may be
compared with
\cite{ChristensenHasannasabPhilippStoeva2024Mystery, GallardoGutierrezPartington2026, KrishtalMiller2026Demystifying,KrishtalMiller2025Kadec}.
Closely related diagonal-system formulations also appear in observability and
controllability through vector-valued interpolation; see
\cite{JacobPartingtonPott2007,JacobZwart2001}.  We also use the standard frame and
Riesz-basis facts;  convenient general references are Christensen's \cite {Christensen2016} and Heil's 
\cite{Heil2011} texts.

The main result proved in the construction section is the following.

\begin{theorem}\label{thm:main}
There exists a bounded invertible normal operator $T$ on $\ell^2(\mathbb N)$ and a vector
$g\in\ell^2(\mathbb N)$ such that
\[
  \left\{\frac{T^k g}{\|T^k g\|}:k=0,1,2,\ldots\right\}
\]
is a frame for $\ell^2(\mathbb N)$.  %
\end{theorem}

The construction forces the normalized orbit to spend a long interval
of time almost entirely inside one block, where it runs through a Fourier basis.
A subsequence is therefore a small perturbation of an orthonormal basis, while a
Schur-test estimate gives the Bessel bound for the whole sequence. This construction illustrates that the orbits of normal operators provide enough flexibility for frame generation, unlike those of self-adjoint operators, which lack the possibility of spectral rotation.   

The penultimate section of this manuscript records a related perturbation result: an arbitrarily small
diagonal normal perturbation of the operator constructed in Theorem \ref{thm:main} produces an unnormalized
operator-orbit Carleson frame. %

\section{Construction of the  frame orbit }\label{sec:construction}

In a nutshell, the idea of the construction is as follows. We define
\begin{align*}
    H=\bigoplus_{m=1}^\infty H_m, \quad T=  \bigoplus_{m=1}^\infty e^{-\eta_m/2}U_m, 
    \quad g=\bigoplus_{m=1}^\infty e^{\beta_m/2}g_m,
\end{align*}
where $H_m$ is finite dimensional, $U_m$ is unitary on $H_m$ and $g_m\in H_m$. Then
\[
  T^k g
  =\bigoplus_{m=1}^\infty e^{(\beta_m-\eta_mk)/2} U_m^{k}\, g_m.
\]
The key to the proof is to design $g_m$ and $U_m$ so that $U_m^kg_m$ contains all elements of a basis of $H_m$ for $k$ in a discrete time interval $J_m$, and to pick $\beta_m$ and $\eta_m$ so that $e^{(\beta_m-\eta_mk)/2}$ is large for $k\in J_m$ and close to 0 for $k\in J_n$ with $n\neq m$. The set of the $T^kg/\|T^kg\|$  with $k\in\bigcup_{m\in\mathbb N}J_m$ is then shown to be a Riesz basis of $H$ and establishing a Bessel bound for all $T^kg/\|T^kg\|$ completes the proof.

The construction of $H$, $T$, and $g$ and the proof of Theorem~\ref{thm:main} are elementary in nature, but also quite intricate and non trivial.  To aid the reader, we include a table of all the constants and mathematical objects that are introduced below.

\begin{longtable}{@{}p{0.12\textwidth}p{0.35\textwidth}p{0.44\textwidth}@{}}
\caption{Constants and notation used in the construction.}\label{tab:constants}\\
\toprule
Symbol & Definition & Role \\
\midrule
\endfirsthead
\toprule
Symbol & Definition & Role \\
\midrule
\endhead
\bottomrule
\endfoot
$Q$ & $100$ & Fixed separation parameter. \\
$N_m$ & $Q^{2m}$ & Dimension of the $m$-th block $H_m$ of $H$. \\
$L_m$ & $3N_m$ & Time  interval length assigned to block $m$. \\
$S_m$ & $\sum_{q=1}^m L_q$, with $S_0=0$ & Right endpoints of the integer intervals. \\
$\tau_m$ & $S_m-\frac12$ & Intersection point of the lines $\ell_m$ and $\ell_{m+1}$. \\
$\eta_m$ & $\sum_{q=m}^\infty Q^{-q}=\frac{Q^{1-m}}{Q-1}$ & Decay exponent for the $m$-th spectral radius. \\
$r_m$ & $e^{-\eta_m/2}$ &  $m$-th spectral radius. \\
$\omega_m$ & $e^{2\pi i/N_m}$ & Primitive $N_m$-th root of unity. \\
$U_m$ & $U_m e_{m,j}=\omega_m^j e_{m,j}$ & Unitary  operator on the $m$-th block. \\
$T$ & $T=\bigoplus_m r_m U_m$ & Constructed bounded normal operator. \\
$\beta_m$ & $\beta_1=0$, $\displaystyle \beta_m=-\sum_{r=1}^{m-1}Q^{-r}\tau_r$ & Exponent in the 
weight of $g$ in  block $m$. \\
$g_m$ & $N_m^{-1/2}\sum_{j=0}^{N_m-1}e_{m,j}$ & Normalized constant vector in $H_m$. \\
$g$ & $\bigoplus_m e^{\beta_m/2}\,g_m$ & Seed vector for the orbit. \\
$\ell_m(t)$ & $\beta_m-\eta_m t$ & Affine line controlling the energy at time $t$ within the $m$-th block at time $t$. \\
$I_m$ & $\{S_{m-1},S_{m-1}+1,\ldots,S_m-1\}$ & Time interval that $\ell_m$ dominates. \\
$J_m$ & $\{S_{m-1}+N_m,\ldots,S_{m-1}+2N_m-1\}$ & Middle third of $I_m$. \\ 
$p_{m,k}$ & 
$\displaystyle
\frac{e^{\ell_m(k)}}{\sum_{n=1}^\infty e^{\ell_n(k)}}$ & Energy of $k$-th sequence vector  in block $m$. \\
$\varepsilon_0$ & $6\sum_{m\ge1}Q^{2m}e^{-Q^m}<10^{-38}$ & Perturbation-size bound for the Riesz-basis subsequence. 
\\
$\rho_m$ &  $r_m$ small $m$ and $\sqrt{1-e^{\beta_m}/N_m}$ on the tail & Perturbed radius used for the unnormalized Carleson-frame orbit. \\
$S$ & $S e_{m,j}=\rho_m\omega_m^j e_{m,j}$ & Perturbed diagonal normal operator. \\
$\lambda_{m,j}$ & $\rho_m\omega_m^j$ & Eigenvalues of the perturbed operator $S$. \\
\end{longtable}

\subsection{Construction of the Hilbert space and of the bounded normal operator}\label{subsec:definition}

We set   $Q=100$ and $N_m=Q^{2m}$ for $m\ge 1$ as described in   Table~\ref{tab:constants}.  As Hilbert space 
we choose
\[
  H=\bigoplus_{m=1}^\infty H_m, \qquad H_m=\mathbb C^{N_m},
\]
with  inner product and norm induced by the summands.
The space $H$ is naturally isomorphic to $\ell^2(\mathbb N)$.
Denoting the standard euclidean basis of $H_m$ by $\{e_{m,j}:0\le j<N_m\}$, we define the diagonal and  unitary operator $U_m$ on $H_m$ by setting 
\[U_m e_{m,j}=\omega_m^j e_{m,j},
   \qquad \mbox{where} \qquad \omega_m=e^{2\pi i/N_m}
\]
is the canonical $N_m$-th principle root of  unity.
With
\(
  \eta_m=\sum_{q=m}^\infty Q^{-q}=\frac{Q^{1-m}}{Q-1}
 \) and \( r_m=e^{-\eta_m/2},
\)
 we finally define  
$T: H\to H$ as
\[
  T=\bigoplus_{m=1}^\infty r_m U_m,\ \text{ that is, } \quad T e_{m,j}=r_m\omega_m^j e_{m,j},   \ m\in\mathbb N,\ j = 0, 1, \ldots, N_m{-}1.
\]
Note that $T$ is diagonal and, therefore, normal. In addition, $T$ is bounded since
\(
  \|T\|=\sup_{m\ge 1}r_m=\lim_{m\to \infty}r_m=1,
\)
and $T$ is invertible since the eigenvalues of $T$ are contained in $[r_1,1)$ with $r_1>0$ since $0<r_1<r_2< \ldots <1$.

\subsection{Construction of the orbit generating vector}

We now choose the vector $g$. To this end, we set $L_m=3N_m$, 
$S_m=\sum_{q=1}^m L_q$ with $S_0=0$ and 
$\tau_m=S_m{-}\frac12$.
Further, we let 
\[
  \beta_1=0\quad \mbox{and}\quad\beta_{m+1}=\beta_m-Q^{-m}\tau_m=-\sum_{r=1}^{m}Q^{-r}\tau_r,\qquad m\ge 1,
\]
so that
\[\beta_m
  \leq -\sum_{r=1}^{m-1}Q^{-r}Q^{2r}
  =-\sum_{r=1}^{m-1}Q^r
  =-\frac{Q^m-Q}{Q-1}  \to -\infty \text{ as }  m\to \infty.
\]
Finally, we set 
\[
  g=\bigoplus_{m=1}^\infty e^{\beta_m/2}\,g_m, \quad \mbox{where}\quad g_m=\frac1{\sqrt{N_m}}\sum_{j=0}^{N_m-1}e_{m,j}, \quad m\geq 1,
\]
and observe that $g\in H$ since 
\ \(\displaystyle 
  \|g\|^2=\sum_{m=1}^\infty e^{\beta_m}\le\sum_{m=1}^\infty e^{-\frac{Q^m-Q}{Q-1}}<2<\infty.
\)

\subsection{The orbit and its normalization}\label{subsec:normalized-formula}

For $k\ge 0$, define
\[
  v_{m,k}=U_m^k g_m=\frac1{\sqrt{N_m}}\sum_{j=0}^{N_m-1}\omega_m^{jk}e_{m,j}\in \mathbb C^{N_m},
\]
that is, \(
  \{v_{m,0},v_{m,1},\ldots,v_{m,N_m-1}\}
\)
is the  Fourier orthonormal basis of $H_m$. Note that 
for fixed $m$, the vector $v_{m,k}$ depends only on $k$ modulo $N_m$, and any set of $N_m$ consecutive values of $k$ lead to the same Fourier basis. More generally,
\[
  \langle v_{m,k},v_{m,\ell}\rangle
  =
  \begin{cases}
    1, & k\equiv \ell \pmod {N_m},\\
    0, & k\not\equiv \ell \pmod {N_m}.
  \end{cases}
\]

With this notation and recalling  $r_m^{2k}=e^{-\eta_m k}$, we obtain
\[
  T^k g
  =\bigoplus_{m=1}^\infty e^{\beta_m/2}\,r_m^k v_{m,k}=\bigoplus_{m=1}^\infty e^{\beta_m/2}\,e^{-\eta_m k/2} v_{m,k}=\bigoplus_{m=1}^\infty e^{(\beta_m-\eta_mk)/2} v_{m,k}=\bigoplus_{m=1}^\infty e^{\ell_m(k)/2} v_{m,k},
\]
where the crucial affine linear maps
\[
  \ell_m(t)=\beta_m-\eta_m t, \ m\geq 1,
\]
are illustrated in Figure~\ref{fig:first-five-lines}.
Since
\(\displaystyle
  \|T^k g\|^2
  =\sum_{m=1}^\infty e^{\ell_m(k)},
\)
the energy of the $m$-th summand appearing in $T^kg/\|T^k g\|$ is 
\begin{equation}
\label{pmk}
  p_{m,k}
  =\frac{e^{\ell_m(k)}}{\sum_{n=1}^\infty e^{\ell_n(k)}}.
\end{equation}
Clearly $p_{m,k}>0$ and $\sum_m p_{m,k}=1$. In summary, the elements of the normalized orbit  are  given by
\[
  x_k:=\frac{T^k g}{\|T^k g\|}
  =\bigoplus_{m=1}^\infty \sqrt{p_{m,k}}\,v_{m,k}.
\]

\begin{figure}
\centering
\begin{tikzpicture}[x=1.55cm,y=0.18cm,>=stealth,font=\small]
  \def\xmin{-0.22}
  \def\xmax{5.42}
  \def\ymin{-32.5}
  \def\ymax{1.8}
  \def\Jthreeleft{2.333333333}
  \def\Jthreeright{2.666666667}

  \fill[cyan!10] (\Jthreeleft,\ymin) rectangle (\Jthreeright,\ymax);
  \draw[gray!40] (\xmin,\ymin) rectangle (\xmax,\ymax);
  \foreach \yy in {-30,-25,-20,-15,-10,-5,0} {
    \draw[gray!18] (\xmin,\yy) -- (\xmax,\yy);
  }
  \foreach \xx in {1,2,3,4} {
    \draw[gray!55,densely dashed] (\xx,\ymin) -- (\xx,\ymax);
  }
  \draw[cyan!50!black,densely dotted] (\Jthreeleft,\ymin) -- (\Jthreeleft,\ymax);
  \draw[cyan!50!black,densely dotted] (\Jthreeright,\ymin) -- (\Jthreeright,\ymax);

  \draw[->,thick] (\xmin,0) -- (6.10,0) node[anchor=west] {$t$};
  \draw[->,thick] (0,\ymin) -- (0,2.65) node[anchor=south] { };

  \begin{scope}
    \clip (\xmin,\ymin) rectangle (\xmax,\ymax);
    \draw[very thick,blue!70!black]   plot[domain=0:5.42,samples=2] (\x,{  0-10*\x});
    \draw[very thick,orange!85!black] plot[domain=0:5.42,samples=2] (\x,{ -2- 8*\x});
    \draw[very thick,green!55!black]  plot[domain=0:5.42,samples=2] (\x,{ -6- 6*\x});
    \draw[very thick,red!75!black]    plot[domain=0:5.42,samples=2] (\x,{-12- 4*\x});
    \draw[very thick,purple!75!black] plot[domain=0:5.42,samples=2] (\x,{-20- 2*\x});
  \end{scope}

  \foreach \xx/\lab in {1/{\tau_1},2/{\tau_2},3/{\tau_3},4/{\tau_4}} {
    \draw[thick] (\xx,0.55) -- (\xx,-0.55);
    \node[anchor=south,yshift=3pt,fill=white,inner sep=1pt] at (\xx,0) {$\lab$};
  }
  \fill[black] (1,-10) circle (1.8pt);
  \fill[black] (2,-18) circle (1.8pt);
  \fill[black] (3,-24) circle (1.8pt);
  \fill[black] (4,-28) circle (1.8pt);

  \draw[cyan!70!black,line width=1.2pt] (\Jthreeleft,-1.45) -- (\Jthreeright,-1.45);
  \draw[cyan!70!black,line width=1.2pt] (\Jthreeleft,-0.95) -- (\Jthreeleft,-1.95);
  \draw[cyan!70!black,line width=1.2pt] (\Jthreeright,-0.95) -- (\Jthreeright,-1.95);
  \node[anchor=north,fill=white,inner sep=1.5pt,text=cyan!55!black]
    at (2.5,-2.15) {$J_3$};

  \draw[thick] (-0.06,0) -- (0.06,0);
  \node[anchor=east,xshift=-3pt,text=blue!70!black,fill=white,inner sep=1pt] at (0,0) {$\beta_1$};
  \draw[thick] (-0.06,-2) -- (0.06,-2);
  \node[anchor=east,xshift=-3pt,text=orange!85!black,fill=white,inner sep=1pt] at (0,-2) {$\beta_2$};
  \draw[thick] (-0.06,-6) -- (0.06,-6);
  \node[anchor=east,xshift=-3pt,text=green!55!black,fill=white,inner sep=1pt] at (0,-6) {$\beta_3$};
  \draw[thick] (-0.06,-12) -- (0.06,-12);
  \node[anchor=east,xshift=-3pt,text=red!75!black,fill=white,inner sep=1pt] at (0,-12) {$\beta_4$};
  \draw[thick] (-0.06,-20) -- (0.06,-20);
  \node[anchor=east,xshift=-3pt,text=purple!75!black,fill=white,inner sep=1pt] at (0,-20) {$\beta_5$};

  \draw[very thick,blue!70!black]   (5.72,-5.0) -- (6.06,-5.0);
  \node[anchor=west] at (6.12,-5.0) {$\ell_1$};
  \draw[very thick,orange!85!black] (5.72,-8.0) -- (6.06,-8.0);
  \node[anchor=west] at (6.12,-8.0) {$\ell_2$};
  \draw[very thick,green!55!black]  (5.72,-11.0) -- (6.06,-11.0);
  \node[anchor=west] at (6.12,-11.0) {$\ell_3$};
  \draw[very thick,red!75!black]    (5.72,-14.0) -- (6.06,-14.0);
  \node[anchor=west] at (6.12,-14.0) {$\ell_4$};
  \draw[very thick,purple!75!black] (5.72,-17.0) -- (6.06,-17.0);
  \node[anchor=west] at (6.12,-17.0) {$\ell_5$};

  \node[anchor=north] at (2.65,-34.1)
    {compressed $t$-axis};
\end{tikzpicture}%
\caption{Schematic of the first five affine functions $\ell_m(t)$.  The labels $\tau_1,\ldots,\tau_4$ are placed on the $t$-axis, while the intercepts $\beta_i=\ell_i(0)$ are marked on the vertical axis.  The highlighted interval $J_3$ is the schematic middle third between $\tau_2$ and $\tau_3$. 
The actual values of $\tau_1$, $\tau_2$, $\tau_3$,  $\tau_4$, $\ldots$ grow too rapidly for a good illustration, so the $t$-axis is compressed and the $\tau_m$
are placed equidistantly.%
}
\label{fig:first-five-lines}
\end{figure}
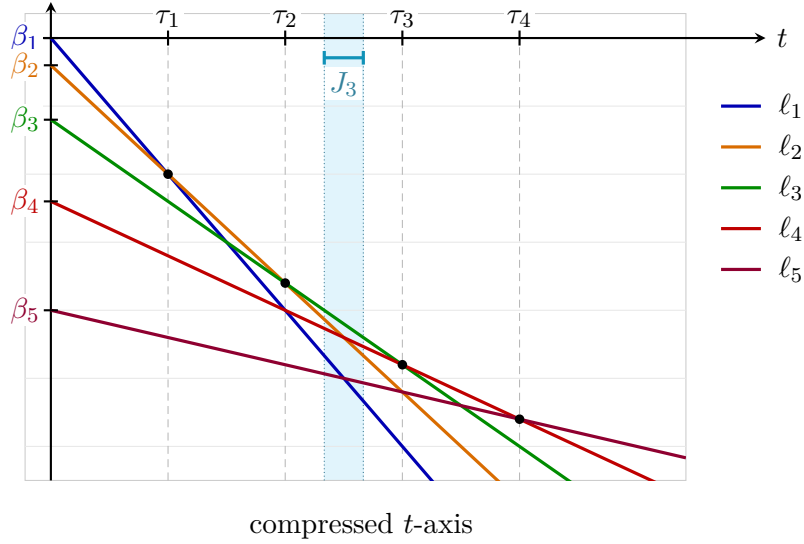

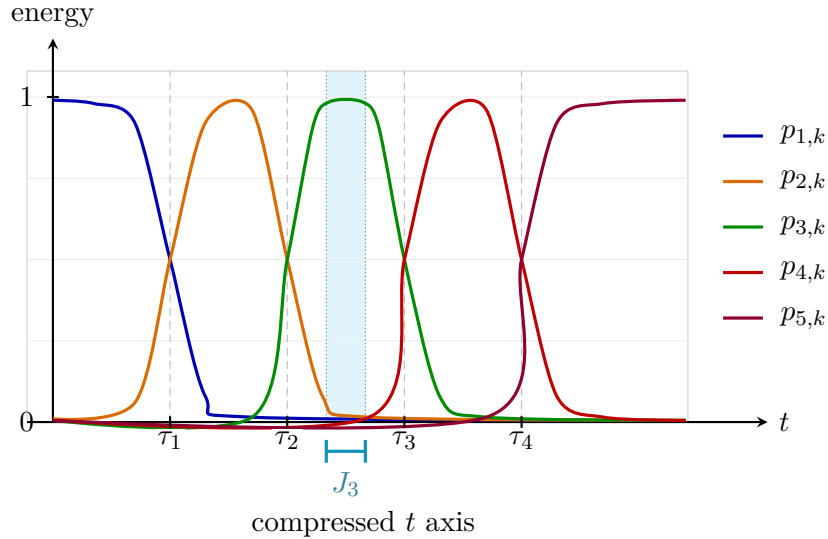
\begin{figure}[H]
\centering
\begin{tikzpicture}[x=1.55cm,y=4.3cm,>=stealth,font=\small
]
  \def\xmin{-0.22}
  \def\xmax{5.42}
  \def\Jthreeleft{2.333333333}
  \def\Jthreeright{2.666666667}

  \fill[cyan!10] (\Jthreeleft,0) rectangle (\Jthreeright,1.08);
  \draw[gray!40] (\xmin,0) rectangle (\xmax,1.08);
  \foreach \yy in {0.25,0.50,0.75,1.00} {
    \draw[gray!16] (\xmin,\yy) -- (\xmax,\yy);
  }
  \foreach \xx in {1,2,3,4} {
    \draw[gray!55,densely dashed] (\xx,0) -- (\xx,1.08);
  }
  \draw[cyan!50!black,densely dotted] (\Jthreeleft,0) -- (\Jthreeleft,1.08);
  \draw[cyan!50!black,densely dotted] (\Jthreeright,0) -- (\Jthreeright,1.08);

  \draw[->,thick] (\xmin,0) -- (6.10,0) node[anchor=west] {$t$};
  \draw[->,thick] (0,0) -- (0,1.18) node[anchor=south] {energy};
  \draw[thick] (-0.05,1) -- (0.05,1);
  \node[anchor=east,xshift=-3pt] at (0,1) {$1$};
  \node[anchor=east,xshift=-3pt] at (0,0) {$0$};

  \draw[very thick,blue!70!black]
    plot[smooth] coordinates {(0,0.99) (0.35,0.98) (0.70,0.92) (1.00,0.50) (1.30,0.08) (1.70,0.015) (5.40,0.005)};
  \draw[very thick,orange!85!black]
    plot[smooth] coordinates {(0,0.01) (0.70,0.06) (1.00,0.50) (1.30,0.92) (1.70,0.96) (2.00,0.50) (2.30,0.08) (2.70,0.015) (5.40,0.005)};
  \draw[very thick,green!55!black]
    plot[smooth] coordinates {(0,0.005) (1.70,0.015) (2.00,0.50) (2.22,0.91) (2.36,0.985) (2.64,0.985) (2.78,0.91) (3.00,0.50) (3.30,0.08) (3.70,0.015) (5.40,0.005)};
  \draw[very thick,red!75!black]
    plot[smooth] coordinates {(0,0.005) (2.70,0.015) (3.00,0.50) (3.30,0.92) (3.70,0.96) (4.00,0.50) (4.30,0.08) (4.70,0.015) (5.40,0.005)};
  \draw[very thick,purple!75!black]
    plot[smooth] coordinates {(0,0.005) (3.70,0.015) (4.00,0.50) (4.30,0.92) (4.70,0.98) (5.40,0.99)};

  \foreach \xx/\lab in {1/{\tau_1},2/{\tau_2},3/{\tau_3},4/{\tau_4}} {
    \draw[thick] (\xx,0.018) -- (\xx,-0.018);
    \node[anchor=north,fill=white,inner sep=1pt] at (\xx,-0.02) {$\lab$};
  }

  \draw[cyan!70!black,line width=1.2pt] (\Jthreeleft,-0.09) -- (\Jthreeright,-0.09);
  \draw[cyan!70!black,line width=1.2pt] (\Jthreeleft,-0.055) -- (\Jthreeleft,-0.125);
  \draw[cyan!70!black,line width=1.2pt] (\Jthreeright,-0.055) -- (\Jthreeright,-0.125);
  \node[anchor=north,fill=white,inner sep=1.5pt,text=cyan!55!black]
    at (2.5,-0.14) {$J_3$};

  \draw[very thick,blue!70!black]   (5.72,0.88) -- (6.06,0.88);
  \node[anchor=west] at (6.12,0.88) {$p_{1,k}$};
  \draw[very thick,orange!85!black] (5.72,0.74) -- (6.06,0.74);
  \node[anchor=west] at (6.12,0.74) {$p_{2,k}$};
  \draw[very thick,green!55!black]  (5.72,0.60) -- (6.06,0.60);
  \node[anchor=west] at (6.12,0.60) {$p_{3,k}$};
  \draw[very thick,red!75!black]    (5.72,0.46) -- (6.06,0.46);
  \node[anchor=west] at (6.12,0.46) {$p_{4,k}$};
  \draw[very thick,purple!75!black] (5.72,0.32) -- (6.06,0.32);
  \node[anchor=west] at (6.12,0.32) {$p_{5,k}$};

  \node[anchor=north] at (2.65,-0.24)
    {compressed $t$ axis};
\end{tikzpicture}%
\caption{Schematic behavior of the block weights $p_{m,k}$.  The weights are the soft-max weights associated with the affine functions $\ell_m(k)$: on the middle third $J_3$, the third affine function dominates strongly, so $p_{3,k}$ is close to $1$ and the other weights are close to $0$. Note that schematic weights are displayed; the actual transition regions are much sharper for $Q=100$.}
\label{fig:weights}
\end{figure}

Figure~\ref{fig:weights} illustrates the  behavior of the weights $p_{m,k}$.  It
is the soft-max version of the affine-line picture in Figure \ref{fig:first-five-lines}: when one line is well above
all the others, the corresponding weight is close to one, and the other weights
are close to zero.  The transition between dominant blocks occurs near the
crossing points $\tau_m$ that are defined below. The crossings are avoided by focusing on the middle thirds, e.g., the interval $J_3\subseteq I_3$ that is displayed in the figures.

\subsection{Analysis of the weights and their dominance in middle thirds}\label{subsec:dominance}

Crucial for the understanding of the weights $p_{m,k}$ are the affine linear functions 
\(
  \ell_m(t)=\beta_m-\eta_m t
\) defined above. 
Note that the crossing points of these lines are  exactly at the prescribed numbers $\tau_m=S_m{-}\frac 12$.  Indeed, we have
\[\ell_{m+1}(\tau_m) =\beta_{m+1}-\eta_{m+1}\tau_m =
\beta_{m}-Q^{-m}\tau_m -\eta_{m+1}\tau_m
=
\beta_{m}-\eta_{m}\tau_m
=\ell_m(\tau_m).
\]
The slopes are $-\eta_m$, and because $\eta_1>\eta_2>\eta_3> \ldots>0$, the lines become
flatter as $m$ increases. Thus, the upper envelope passes from $\ell_1$ to
$\ell_2$ at $\tau_1$, from $\ell_2$ to $\ell_3$ at $\tau_2$, and so on.
This geometry is illustrated schematically in Figure \ref{fig:first-five-lines}.  The actual values for $Q=100$ are too far apart to make a useful literal drawing, so the horizontal axis is compressed and
successive crossing points $\tau_1$, $\tau_2$, $\tau_3$, and $\tau_4$
are placed equidistantly.  In this compressed drawing, the interval $J_3$ is
shown as the middle third of the interval between $\tau_2$ and
$\tau_3$.

Since 
\[
  S_{m}{-}1<\tau_{m}=S_{m}{-}\frac 1 2<S_m, \quad m\geq 1,
\]
the line $\ell_m$ dominates all other lines exactly on $I_m =\{S_{m-1},S_{m-1}{+}1,\ldots,S_m{-}1\}$. The middle third of $I_m$ is denoted by 
$J_m=\{S_{m-1}+N_m,\ldots,S_{m-1}+2N_m-1\}$, all points in $J_m$ are therefore at least at distance $N_m{+}\frac 1 2$ from the nearest crossing points $\tau_m$ and $\tau_{m+1}$.
These  distance estimates are the source of the exponentially small error
terms in Lemma~\ref{lem:dominance} below.

\begin{lemma}\label{lem:dominance}
For every $m\ge 1$ and every $k\in J_m$ we have \ 
\(\displaystyle 
  p_{m,k}\ge 1-3e^{-Q^m}.
\)
\end{lemma}

\begin{proof}
We have
\begin{align*}
1- p_{m,k}&=1-\frac{e^{\ell_m(k)}}{\sum_{n=1}^\infty e^{\ell_n(k)}}=1-\frac{1}{1+\sum_{n\neq m} e^{\ell_n(k)-\ell_m(k)}}
=
 \frac{\sum_{n\neq m} e^{\ell_n(k)-\ell_m(k)}}{1+\sum_{n\neq m} e^{\ell_n(k)-\ell_m(k)}}\\ &\leq   \sum_{n\neq m} e^{\ell_n(k)-\ell_m(k)},
\end{align*}
so it suffices to show
\[R_{m,k}=  \sum_{n\neq m} e^{\ell_n(k)-\ell_m(k)}\leq 3e^{-Q^m}.\]

Recalling the adjacent-line identity 
\(
  \ell_m(\tau_m)=\ell_{m+1}(\tau_m)
\)
and observing that
\(
  \ell_m(t)-\ell_{m+1}(t)
\)
is affine with slope
\(
  -\eta_m-(-\eta_{m+1})=-(\eta_m-\eta_{m+1})=-Q^{-m},
\)
we have, for every real $t$,
\begin{equation}\label{lmdif}
  \ell_m(t)-\ell_{m+1}(t)=Q^{-m}(\tau_m-t).
\end{equation}

Let $k\in J_m$. We first consider the case $n=m+s$ with $s\ge1$.
Using 
\eqref{lmdif} repeatedly in a telescoping sum gives
\[
\begin{aligned}
  \ell_m(k)-\ell_{m+s}(k)
  &=\sum_{r=m}^{m+s-1}\bigl(\ell_r(k)-\ell_{r+1}(k)\bigr)  
&=\sum_{r=m}^{m+s-1}Q^{-r}\left(\tau_r-k\right).
\end{aligned}
\]
Since 
all points in $J_m$ are  at least at distance $N_m{+}\frac 1 2$ from $\tau_{m-1}$ and  $\tau_{m}$,  the first summand satisfies
\[
  Q^{-m}\left(\tau_m-k\right)
  \ge Q^{-m}N_m
  =Q^{-m}Q^{2m}
  =Q^m.
\]
For $r\ge m+1$, we  have
\[
  \tau_r-k\ge \tau_r-\tau_m=S_r-S_m
  =\sum_{a=m+1}^r L_a
  \ge L_r=3N_r=3Q^{2r}.
\]
Thus,
\[
  Q^{-r}\left(\tau_r-k\right)
  \ge Q^{-r}3Q^{2r}=3Q^r.
\]
Using these estimates, we obtain
\[
  \ell_m(k)-\ell_{m+s}(k)
  \ge Q^m+3\sum_{r=m+1}^{m+s-1}Q^r
  \ge Q^m+3(s-1)Q^{m+1}, 
\]
and  
\begin{align*}
  \sum_{n>m}e^{\ell_n(k)-\ell_m(k)}
  &\le e^{-Q^m}\sum_{s=1}^\infty e^{-3(s-1)Q^{m+1}} 
  =\frac{e^{-Q^m}}{1-e^{-3Q^{m+1}}}
  \le 2e^{-Q^m},
\end{align*}
where we simply used that $Q=100$ implies
$e^{-3Q^{m+1}}<1/2$.

Next, we estimate the left-hand tail $n<m$, again using \eqref{lmdif}.  For $m=1$ the left hand tail  is empty.  For
$m\ge2$ and $n<m$ we have
\[
\begin{aligned}
  \ell_m(k)-\ell_n(k)
  &=\sum_{r=n}^{m-1}\bigl(\ell_{r+1}(k)-\ell_r(k)\bigr) 
  =\sum_{r=n}^{m-1}Q^{-r}\left(k-\tau_r \right)
  \ge Q^{-(m-1)}\left(k-\tau_{m-1}\right)\\
  &\ge Q^{-(m-1)}N_m
  =Q^{-(m-1)}Q^{2m}
  =Q^{m+1}.
  \end{aligned}
\]
Therefore
\[
  \sum_{n<m}e^{\ell_n(k)-\ell_m(k)}
  \le (m-1)e^{-Q^{m+1}}
  \le e^{-Q^m}.
\]
The final inequality holds   trivially for $m=1$
and for $m\ge2$ it follows from 
\[
(m-1)e^{-Q^{m+1}}e^{Q^m}=(m-1)
e^{-Q^m(Q-1)}<(m-1)
e^{-Q^m}<1
\]
since $Q$ is large.

Combining the two tails yields
\(\displaystyle 
  R_{m,k}
  =\sum_{n\ne m}e^{\ell_n(k)-\ell_m(k)}
  \le 3e^{-Q^m}.
\)
\end{proof}

\subsection{A Riesz-basis subsequence}\label{subsec:riesz-subsequence}

For $k\in J_m$, let $y_k \in H$ be the vector $v_{m,k}$ placed in the $m$-th block
and zero in every other block.  Because $J_m$ consists of exactly $N_m$
consecutive integers, every residue of $k$ modulo $N_m$ appears exactly once. Hence,
\[
  \{y_k\in H:\ k\in J_m\}
\]
is an orthonormal basis for $H_m$ and, therefore,
\[
  \mathcal Y:=\{y_k\in H:\ k\in J:=\bigcup_{m=1}^\infty J_m\}
\]
is an orthonormal basis for $H$.

Clearly, for any $k\in J_m$, we have
\[
  \langle x_k,y_k\rangle=\sqrt{p_{m,k}},
\]
and, therefore, Lemma~\ref{lem:dominance} implies
\[
\begin{aligned}
  \|x_k-y_k\|^2
  &=\|x_k\|^2+\|y_k\|^2-2\operatorname{Re}\langle x_k,y_k\rangle 
  =2(1-\sqrt{p_{m,k}})
  \le 2(1-p_{m,k})
  \le 6e^{-Q^m}.
\end{aligned}
\]
 Thus,
\[
  \sum_{k\in J}\|x_k-y_k\|^2
  \le 6\sum_{m=1}^\infty N_m e^{-Q^m}
  =6\sum_{m=1}^\infty Q^{2m}e^{-Q^m}
  =:\varepsilon_0.
\]
For $Q=100$, $\varepsilon_0<10^{-38}<1$. Hence, the standard frame-perturbation criterion  implies that  $\{x_k\}_{k\in J}$ is a Riesz basis for $H$ with lower frame bound $A_J=(1-\sqrt{\varepsilon_0})^2$ (see \cite[Corollary 22.1.5 and Exercise 22.4]{Christensen2016}). Indeed,
 for any finitely supported sequence $\{c_k\}_{k\in J}$, we have
\[
\left\|\sum_{k\in J} c_k(x_k-y_k)\right\|
\le\sum_{k\in J} |c_k|^2\left\|(x_k-y_k)\right\|^2
\le\left(\sum_{k\in J} |c_k|^2\right)^{\frac 1 2}\!\!
\left(\sum_{k\in J} \left\|  x_k-y_k\right\|^2\right)^{\frac 1 2}\!\!\le \epsilon_0
\left(\sum_{k\in J} |c_k|^2\right)^{\frac 1 2}\!\!.
\]
This criterion goes back to Paley and Wiener 
\cite[p.~100]{PaleyWiener1934};
see also \cite[Theorem~13, p.~35]{Young2001}, \cite{CC3} and \cite{CH}.

\subsection{The Bessel bound for the frame sequence}\label{subsec:bessel-bound}

It remains to check that the full sequence $\{x_k:k\ge 0\}$ has a finite upper
frame bound.  

\begin{lemma}\label{lem:bessel-estimates}
For $Q=100$, the following two estimates hold:
\begin{equation}
    \label{sumest}
  \sup_{k\ge0}\sum_{m=1}^\infty \sqrt{p_{m,k}}\le 4
\end{equation}
and
\begin{equation}
    \label{progest}
  \sup_{m\ge1}\sup_{0\le j<N_m}
  \sum_{q\ge0}\sqrt{p_{m,j+qN_m}}\le 7.
\end{equation}
\end{lemma}

\begin{proof} For a fixed $m$,
we begin with two elementary one-sided estimates.  If $k\ge S_m$, then the
term $e^{\ell_{m+1}(k)}$ appears in the denominator defining $p_{m,k}$ (see \eqref{pmk}), so we can use \eqref{lmdif} once more to obtain
\begin{equation}\label{foe}
     p_{m,k}
  \le e^{\ell_m(k)-\ell_{m+1}(k)}
  =e^{Q^{-m}\left(\tau_m-k\right)}
  =e^{-Q^{-m}(k-\tau_m)}.
\end{equation}
Similarly, if $m\ge2$ and $k<S_{m-1}$, then comparison with the previous line
gives
\begin{equation}\label{soe}
  p_{m,k}
  \le e^{\ell_m(k)-\ell_{m-1}(k)}
  =e^{-Q^{-(m-1)}(\tau_{m-1}-k)}.
\end{equation}
For $S_{m-1}\leq k<S_m$, that is, inside $I_m$,  we shall  use the trivial estimate $p_{m,k}\le1$.

We first prove \eqref{sumest}, i.e., a the uniform bound in $k$.  Fix $k\ge0$, and choose the unique
$M$ such that $k\in I_M$.  The three possible adjacent indices
$M-1,M,M+1$ contribute each at most 1, so in total at most $3$, to $\sum_m\sqrt{p_{m,k}}$.

If $m\le M-2$, then $k$ lies to the right of $I_m$, and in fact
$k\ge S_{m+1}$.  Hence,
\[
  k-\tau_m\ge S_{m+1}-S_m+\frac12=L_{m+1}+\frac12
  \ge 3Q^{2m+2}.
\]
Using the first one-sided estimate \eqref{foe},
\[
  \sqrt{p_{m,k}}
  \le e^ {-\frac12Q^{-m}3Q^{2m+2}}
  =e^{-\frac32Q^{m+2}}.
\]
If $m\ge M+2$, then $k$ lies to the left of $I_m$, and
$k\le S_M-1$. Thus,
\[
  \tau_{m-1}-k
  \ge \tau_{m-1}-S_M+1
  \ge L_{m-1}=3Q^{2m-2}.
\]
The second one-sided estimate \eqref{soe} gives
\[
  \sqrt{p_{m,k}}
  \le e^{-\frac12Q^{-(m-1)}3Q^{2m-2}}
  =e^ {-\frac32Q^{m-1}}.
\]
Therefore,
\[
\begin{aligned}
  \sum_{m=1}^\infty \sqrt{p_{m,k}}
  &\le 3
    +\sum_{m\le M-2}e^{-\frac32Q^{m+2}}
    +\sum_{m\ge M+2}e^{-\frac32Q^{m-1}}  
  \le 3+2\sum_{r=1}^\infty e^{-\frac32Q^r}
  <4,
\end{aligned}
\]
since $Q=100$, and the estimate \eqref{sumest} is established.

We now prove the arithmetic-progression estimate \eqref{progest}.  Fix $m$ and a residue
$0\le j<N_m$.  Consider the progression
\[
  j,\ j+N_m,\ j+2N_m,\ldots.
\]
It meets $I_m$, an interval of length $3N_m$, in at most three points.  These
points contribute to summands which add up to less than $3$.

For the part to the right of $I_m$, write its points as $k_0+rN_m$, where
$r=0,1,2,\ldots$, and $k_0\ge S_m$ is the first point in the progression to the
right of $I_m$. 
The first one-sided estimate \eqref{foe} gives
\[
  \sqrt{p_{m,k_0+rN_m}}
  \le e^{-\frac12Q^{-m}(k_0+rN_m-S_m+1/2)}
  \le e^{-rQ^m/2}.
\]
Thus, the right tail is bounded by
\[
  \sum_{r=0}^\infty e^{-rQ^m/2}
  =\frac1{1-e^{-Q^m/2}}
  \le 2.
\]

For the part to the left of $I_m$ there is nothing to prove when $m=1$.
Assume $m\ge2$.  Moving one step to the left in the same residue class changes
$k$ by $-N_m$, so the exponent
\[
  Q^{-(m-1)}(\tau_{m-1}-k)
\]
increases by
\[
  Q^{-(m-1)}N_m=Q^{m+1}.
\]
Using the second one-sided estimate \eqref{soe} in the same way as above, the left tail is
bounded by
\[
  \sum_{r=0}^\infty e^{-rQ^{m+1}/2}
  \le 2.
\]
Combining the inside part, the right tail, and the left tail gives
\[
  \sum_{q\ge0}\sqrt{p_{m,j+qN_m}}
  \le 3+2+2=7,
\]
and the lemma is proved.
\end{proof}

We now apply the Schur test \cite{Christensen2016} to the Gram matrix of the full normalized orbit.
Using the Fourier orthogonality formula in each block, we have
\[
  \langle v_{m,k},v_{m,\ell}\rangle
  =
  \begin{cases}
    1, & k\equiv \ell \pmod {N_m},\\
    0, & k\not\equiv \ell \pmod {N_m}.
  \end{cases}
\]
Therefore the Gram matrix entries of $\{x_k\}$ are
\[
\begin{aligned}
  G_{k\ell}
  &=\langle x_\ell,x_k\rangle 
  =\sum_{m=1}^\infty \sqrt{p_{m,k}p_{m,\ell}}\,
      \langle v_{m,\ell},v_{m,k}\rangle 
  =\sum_{m:\,k\equiv \ell\, (\mathrm{mod}\,N_m)}
      \sqrt{p_{m,k}p_{m,\ell}}.
\end{aligned}
\]
For each fixed $k$,
\[
\begin{aligned}
  \sum_{\ell=0}^\infty |G_{k\ell}|
  &\le \sum_{m=1}^\infty \sqrt{p_{m,k}}
      \sum_{\ell:\,\ell\equiv k\, (\mathrm{mod}\,N_m)}\sqrt{p_{m,\ell}}  
  &\le 4\cdot 7=28
\end{aligned}
\]
by Lemma \ref{lem:bessel-estimates}.
The same estimate holds for the column sums because $G$ is Hermitian.  By the
Schur test, $G$ defines a bounded operator on $\ell^2(\Nzero)$ with norm at
most $28$.  Equivalently, the sequence $\{x_k:k\ge0\}$ is Bessel with Bessel
bound at most $28$.

\subsection{Conclusion of the proof of Theorem~\ref{thm:main}}\label{subsec:main-proof}

The subsequence $\{x_k:k\in J\}$ is a Riesz basis.  Hence it satisfies a lower
frame estimate
\[
  A_J\|h\|^2\le \sum_{k\in J}|\langle h,x_k\rangle|^2,
  \qquad h\in H,
\]
with $A_J>0$.  Since $J\subset\Nzero$, the full sequence satisfies
\[
  A_J\|h\|^2
  \le \sum_{k\in J}|\langle h,x_k\rangle|^2
  \le \sum_{k=0}^\infty |\langle h,x_k\rangle|^2.
\]
Thus, the lower frame bound for the subsequence is also a lower frame bound for
the full sequence.  Lemma~\ref{lem:bessel-estimates} and the Schur test give the
upper frame bound as seen in Section~\ref{subsec:bessel-bound}.  Therefore,
\[
  \left\{\frac{T^k g}{\|T^k g\|}:k=0,1,2,\ldots\right\}
\]
is a frame for $H\cong\ell^2(\mathbb N)$. It can be easily seen from the construction that the system remains a frame if the first vector corresponding to $k=0$ is removed. 
This completes the proof of Theorem~\ref{thm:main}.

\section{Carleson-frame perturbation}\label{sec:carleson-perturbation}

The normalized orbit constructed above is a frame, while the unnormalized orbit
$\{T^k g:k\ge0\}$ is Bessel but not a frame, see Remark~\ref{Remark}.  There is nevertheless an
arbitrarily small diagonal normal perturbation of $T$ for which the
unnormalized orbit of the same vector $g$ is a frame.  This is exactly the
Hardy-space Carleson-kernel mechanism that was first observed in the context of dynamical sampling in \cite{AldroubiCabrelliMolterTang2017}.

Following \cite{AldroubiCabrelliMolterTang2017,  ChristensenHasannasabPhilippStoeva2024Mystery}, we define Carleson frames in the following way.

\begin{definition}
    A complex sequence $\{z_j\}_{j = 0}^\infty$ is a \emph{Carleson sequence} if $z_j \in \DD = \{z\in\CC: |z|<1\}$ and 
\begin{equation}\label{CCond}
\inf_{k} \prod_{j\neq k} \bigg|\frac{z_k - z_j}{1 - z_j\overline{z_k}} \bigg| 
> 0.    
\end{equation}
\end{definition}

\begin{definition}
    Let $\{u_j \}_{j = 0}^\infty$ be an orthonormal basis for a (separable) Hilbert space $\HH$, and assume that $S$ is an operator on $\HH$ satisfying $Su_j = z_ju_j$ for some Carleson sequence $\{z_j\}_{j =0}^\infty$. The system $\{S^{k}\varphi\}_{k = 0}^\infty$ forms a \emph{Carleson frame} if 
\begin{equation}\label{cyclic_vect}
\varphi = \sum_{j = 0}^\infty m_j\sqrt{1-|z_j|^2}\, u_j,
\end{equation}
where $\{m_j\}_{j = 0}^\infty$ is a complex sequence such that $ 0 < A \leq |m_j| \leq B$ for some $A, B \in\RR$.
\end{definition}

The fact that systems in the above definition are frames for $\HH$ was established in \cite[Theorem 3.16]{AldroubiCabrelliMolterTang2017}.

\begin{theorem}
Let $T$   and $g$ be as in Theorem 
\ref{thm:main}.
For every $\varepsilon>0$ there exists a bounded diagonal normal operator $S$ on
$H$ satisfying
\[
  \|S-T\|<\varepsilon
\]
and such that the unnormalized orbit
\[
  \{S^k g:k=0,1,2,\ldots\}
\]
is a Carleson frame for $H$. 
\end{theorem}

\begin{proof}
Choose $M$ so large that, for every $m>M$,
\[
  1-r_m<\frac{\varepsilon}{2}
  \qquad\text{and}\qquad
  1-\sqrt{1-\frac{e^{\beta_m}}{N_m}}<\frac{\varepsilon}{2}.
\]
This is possible because $r_m\to1$ and $e^{\beta_m}/N_m\to0$.  Define
\[
  \rho_m=
  \begin{cases}
    r_m, & m\le M,\\[4pt]
    \displaystyle\sqrt{1-\frac{e^{\beta_m}}{N_m}}, & m>M,
  \end{cases}
\]
and let
\[
  S e_{m,j}=\rho_m\omega_m^j e_{m,j}.
\]
Then $S$ is diagonal and normal.  Since $0<\rho_m<1$ and $\rho_m\to1$, we have
$\|S\|=1$.  Also
\[
  \|S-T\|
  =\sup_{m\ge1}|\rho_m-r_m|<\varepsilon,
\]
because the two operators agree on the first $M$ blocks and both radii are
within $\varepsilon/2$ of $1$ on the remaining blocks.

We will use \cite[Theorem 3.16]{AldroubiCabrelliMolterTang2017} to verify that the orbit 
$\{S^k g:k=0,1,2,\ldots\}$ is a Carleson frame. According to the definitions above, we need to verify  \eqref{CCond} for the eigenvalues $\lambda_{m,j}=\rho_m\omega_m^j$ and  \eqref{cyclic_vect} for the generating vector $g$.

We begin by noticing that the coefficients of $g$ have the correct  size
$\sqrt{1-|\lambda_{m,j}|^2}$ on the tail. Therfore, 
\eqref{cyclic_vect} holds.

Thus, it remains only to show that the eigenvalues  form a Carleson interpolating sequence
\cite{Carleson1958, Seip2004, ShapiroShields1961}, i.e., that for 
  $\lambda_{m,j}=\rho_m\omega_m^j$,
   $0\le j<N_m$, we must have
\begin{equation}
    \label{Cest}
  \inf_{m,j}
  \prod_{(n,\ell)\ne(m,j)}
  \left|
    \frac{\lambda_{m,j}-\lambda_{n,\ell}}
         {1-\overline{\lambda_{n,\ell}}\lambda_{m,j}}
  \right|>0.
\end{equation}

The set $\Lambda=\{\lambda_{m,j}=\rho_m \omega_m^j:0\le j<N_m,\ m\ge1\}$
is a concrete root-of-unity block version of the standard ``zeros on rapidly
approaching concentric circles'' constructions used for thin/interpolating
Blaschke products; compare with the geometric concentric-circle construction of
Gallardo-Gutiérrez and Gorkin \cite{GallardoGorkin2012} and the general
Naftalevič rotation theorem for producing interpolating sequences with
prescribed moduli \cite{AndreevMcNicholl2010}.

The Carleson product in \eqref{Cest} splits naturally into two pieces:
\[
\prod_{(n,\ell)\ne(m,j)}
\left|
\frac{\lambda_{m,j}-\lambda_{n,\ell}}
     {1-\overline{\lambda_{n,\ell}}\lambda_{m,j}}
\right|
=
A_m\prod_{n\ne m}B_{m,n,j},
\]
where \(A_m\) is the contribution from the other points on the same circle, and
\(B_{m,n,j}\) is the contribution from the whole \(n\)-th circle, $n\neq m$.

To estimate $A_m$ and $B_{m,n,j}$, we will use the standard cyclotomic polynomial identities such as
\[
    \prod_{\ell=0}^{N_n-1}(z-\rho_n\omega_n^\ell)=z^{N_n}-\rho_n^{N_n}, \quad z\in \CC.
\]

For points in the same block, we have
\[
\begin{aligned}
  A_m = \prod_{\ell\ne j}
  \left|
    \frac{\rho_m\omega_m^j-\rho_m\omega_m^\ell}
         {1-\rho_m^2\omega_m^{\ell-j}}
  \right|
  &=
  \frac{N_m\rho_m^{N_m-1}(1-\rho_m^2)}
       {1-\rho_m^{2N_m}},
\end{aligned}
\]
which follows from
\[
    \prod_{\ell\ne j}|\rho_m\omega_m^j-\rho_m\omega_m^\ell|
    =N_m\rho_m^{N_m-1}
\quad\mbox{and}\quad
    \prod_{\ell\ne j}|1-\rho_m^2\omega_m^{j-\ell}|
    =\frac{1-\rho_m^{2N_m}}{1-\rho_m^2}.
\]
Therefore, for $m>M$, we get
\[
  A_m = \frac{e^{\beta_m}\rho_m^{N_m-1}}
       {1-\left(1-e^{\beta_m}/N_m\right)^{N_m}},
\]
which tends to $1$, because $e^{\beta_m}\to0$, 
\[
  1-\left(1-\frac{e^{\beta_m}}{N_m}\right)^{N_m}
  \sim e^{\beta_m} \quad \mbox{and} \quad
    \rho_m^{N_m-1}
    =\left(1-\frac{e^{\beta_m}}{N_m}\right)^{(N_m-1)/2}\to 1.
\]
For $m\leq M$ we decrease the lower bound to absorb the finitely many initial blocks, establishing that
the same-block products $A_m$ are uniformly bounded below.

It remains to show that the cross-block products \(\prod\limits_{n\neq m}B_{m,n,j}\) are bounded below uniformly in $m$ and $j$. 
For a fixed
$\lambda\in\DD$ and an entire block of radius $\rho_n$, the cyclotomic polynomial identities
 give
\[
  \prod_{\ell=0}^{N_n-1}
  \left|
    \frac{\lambda-\rho_n\omega_n^\ell}
         {1-\overline{\lambda}\rho_n\omega_n^\ell}
  \right|
  =
  \left|
    \frac{\lambda^{N_n}-\rho_n^{N_n}}
         {1-(\overline{\lambda}\rho_n)^{N_n}}
  \right|.
\]
It follows that 
\[
    B_{m,n,j}
    =
    \left|
    \frac{\lambda_{m,j}^{N_n}-\rho_n^{N_n}}
         {1-(\rho_n\lambda_{m,j})^{N_n}}
    \right|
\]
so that the product over an entire circle becomes a
pseudohyperbolic-type expression involving \(N_n\)-th powers.

Since the phase can only increase the pseudohyperbolic distance\footnote{Indeed,
\[
\left|
\frac{ae^{i\theta}-b}{1-abe^{i\theta}}
\right|^2
=
\frac{a^2+b^2-2ab\cos\theta}
     {1+a^2b^2-2ab\cos\theta},\quad a,b\in(0,1),
\]
and this expression is minimized when \(\cos\theta=1\).}, we have
\[
    B_{m,n,j}
    \ge
    \frac{\left|\rho_{m}^{N_n}-\rho_n^{N_n}\right|}
         {1-(\rho_n\rho_{m})^{N_n}}.
\]
After an elementary computation\footnote{Writing
    $A=1-\rho_m^{N_n}$,
    $B=1-\rho_n^{N_n}$,
\(\Delta=\max\{A,B\}\), and \(\delta=\min\{A,B\}\),
 we have
\[ B_{m,n,j}
    \ge
    \frac{|A-B|}{A+B-AB}
    \ge
    \frac{\Delta-\delta}{\Delta+\delta}
    =
    \frac{1-\delta/\Delta}{1+\delta/\Delta}.
\]}, it follows that
\begin{equation}
    \label{defest}
    1-B_{m,n,j}
    \le
    2\frac{\min\left\{1-\rho_m^{N_n},1-\rho_n^{N_n}\right\}}{\max\left\{1-\rho_m^{N_n},1-\rho_n^{N_n}\right\}}.
\end{equation}
Thus, the cross-circle factor $B_{m,n,j}$ is close to \(1\) once the radial defects
\(1-\rho_m^{N_n}\) and \(1-\rho_n^{N_n}\) are very different.

We proceed to estimate the radial defects. We observe that 
\[
    1-\rho_m^{N_n}
    \asymp
    \min\left\{e^{\beta_m}Q^{2(n-m)},1\right\}
\quad\mbox{and}\quad
1-\rho_n^{N_n}
    \asymp e^{\beta_n}.    
\]

Since for $n<m$ we have that the smaller defect is
\(1-\rho_m^{N_n}\),  \eqref{defest} yields
\[
    1-B_{m,n,j}
    \lesssim
    e^{\beta_m-\beta_n}Q^{-2(m-n)}\le e^{-2Q^n}, \quad n<m.
\]
Since
\[
  \sum_{n=1}^{\infty}e^{-Q^n}<\infty,
\]
it follows that
\[
    \sum_{n< m}\bigl(1-B_{m,n,j}\bigr)
\]
 is bounded uniformly in \(m\).
 For the case $n>m$, we have
\[
    \sum_{n>m}1-B_{m,n,j}
    \lesssim \sum_{n>m}
    \frac{e^{\beta_n}}{\min\{e^{\beta_m}Q^{2(n-m)},1\}} \le \sum_{n=m+1}^\infty e^{\beta_n-\beta_m}Q^{2(m-n)}+\sum_{n=1}^\infty e^{\beta_n},
\]
which is again bounded uniformly in \(m\).

The previous estimates imply that
\begin{equation}\label{eq:sumest}
    \sum_{n\ne m}\bigl(1-B_{m,n,j}\bigr) \le C
\end{equation}
for some $C$ independent of \(m\) and \(j\). Observe that
\[
    \prod_n(1-\varepsilon_n)
    \ge
    \exp\left(-2\sum_n\varepsilon_n\right)
    \qquad
    \text{when }0\le\varepsilon_n\le\frac12.
\]
Letting $\varepsilon_n = 1 - B_{m,n,j}$
we conclude that there is a constant \(c_1>0\) such that
\[
    \prod_{n\ne m}B_{m,n,j}\ge c_1
    \qquad\text{for all }m,j,
\]
as \eqref{eq:sumest} implies that the number of terms $B_{m,n,j}$ which are smaller than $
\frac12$ is bounded above by $2C$, i.e., uniformly in $m$ and $j$. 
Combining this with the same-circle estimate for $A_m$ gives
\[
    \inf_{m,j}
    \prod_{(n,\ell)\ne(m,j)}
    \left|
    \frac{\lambda_{m,j}-\lambda_{n,\ell}}
         {1-\overline{\lambda_{n,\ell}}\lambda_{m,j}}
    \right|>0,
\]
which is the desired Carleson interpolation inequality.
\end{proof}

\begin{remark}\label{Remark}\rm 
The juxtaposition  of the unnormalized orbits of $T$ and $S$ is illuminating. For the original
operator $T$, the radial defect in the $m$-th block is
\[
  1-r_m^2=1-e^{-\eta_m}\asymp Q^{-m},
\]
whereas the squared coefficient of $g$ in each coordinate of the block is
\(
  {e^{\beta_m}}/{N_m}={e^{\beta_m}}/{2Q^m}.
\)
These quantities are not comparable; in fact
\[
  \frac{e^{\beta_m}/N_m}{1-r_m^2}\to0.
\]
The associated weighted  kernels therefore have no lower Riesz bound, so
$\{T^k g\}$ is Bessel but not a frame.  The perturbation replaces the radii by
$\rho_m$ so that
\(
  1-\rho_m^2={e^{\beta_m}}/{N_m}
\)
on the tail, exactly matching the correct normalization.
We illustrate the spectra of the generating operators $T$ and $S$ in the following picture. 
\end{remark}

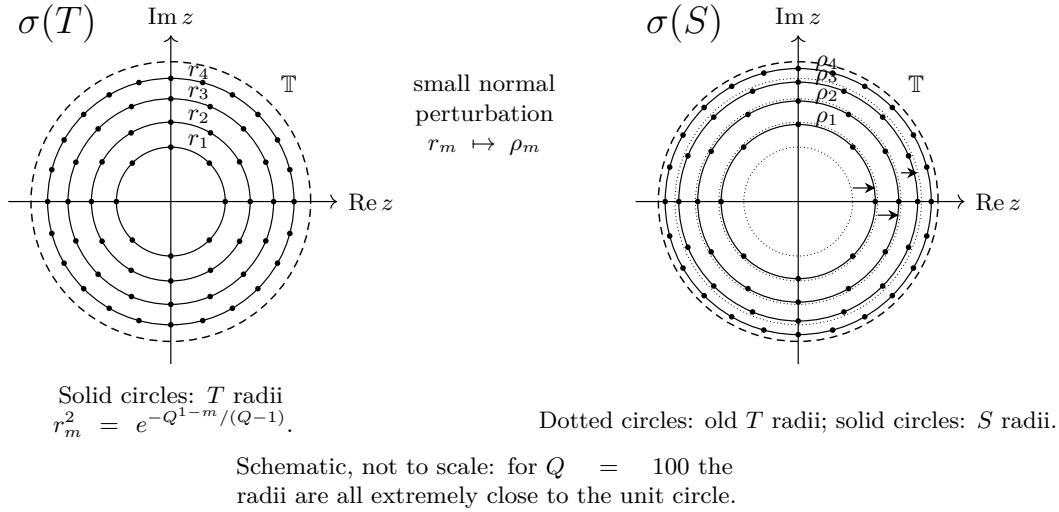
\begin{tikzpicture}[
  scale=1,
  every node/.style={font=\small},
  axis/.style={->, line width=.45pt},
  unitcircle/.style={densely dashed, line width=.55pt},
  ring/.style={line width=.45pt},
  oldring/.style={densely dotted, line width=.35pt},
  specpt/.style={circle, fill=black, inner sep=1.05pt},
  moved/.style={-{Stealth[length=4pt,width=4pt]}, line width=.45pt}
]


\begin{scope}[shift={(-4.15,0)}]
  \node[font=\large] at (-1.5,2.35) {$\sigma(T)$};
  \node[align=center,font=\scriptsize,text width=5.2cm] at (0,2.68)
   {
   };

  \draw[axis] (-2.15,0) -- (2.2,0) node[right, font=\scriptsize] {$\operatorname{Re} z$};
  \draw[axis] (0,-2.15) -- (0,2.2) node[above, font=\scriptsize] {$\operatorname{Im} z$};

  \draw[unitcircle] (0,0) circle (1.85);
  \node[font=\scriptsize, anchor=south west] at (1.31,1.31) {$\mathbb T$};

  \foreach \r/\last/\den/\lab in {0.72/7/8/r_1,1.05/11/12/r_2,1.36/15/16/r_3,1.63/23/24/r_4}{
    \draw[ring] (0,0) circle (\r);
    \node[font=\scriptsize, anchor=west] at (0.08,\r+0.08) {$\lab$};
    \foreach \j in {0,...,\last}{
      \fill ({\r*cos(360*\j/\den)},{\r*sin(360*\j/\den)}) circle (1.05pt);
    }
  }

  \node[align=center,font=\scriptsize,text width=5.0cm] at (0,-2.78)
    {Solid circles: $T$ radii $r_m^2=e^{-Q^{1-m}/(Q-1)}$. 
    };
\end{scope}

\begin{scope}[shift={(4.15,0)}]
  \node[font=\large] at (-1.5,2.35) {$\sigma(S)$};
  \node[align=center,font=\scriptsize,text width=5.2cm] at (0,2.68)
  {
  };

  \draw[axis] (-2.15,0) -- (2.2,0) node[right, font=\scriptsize] {$\operatorname{Re} z$};
  \draw[axis] (0,-2.15) -- (0,2.2) node[above, font=\scriptsize] {$\operatorname{Im} z$};

  \draw[unitcircle] (0,0) circle (1.85);
  \node[font=\scriptsize, anchor=south west] at (1.31,1.31) {$\mathbb T$};

  \foreach \old in {0.72,1.05,1.36,1.63}{
    \draw[oldring] (0,0) circle (\old);
  }

  \foreach \r/\last/\den/\lab in {1.02/7/8/\rho_1,1.33/11/12/\rho_2,1.58/15/16/\rho_3,1.76/23/24/\rho_4}{
    \draw[ring] (0,0) circle (\r);
    \node[font=\scriptsize, anchor=west] at (0.08,\r+0.07) {$\lab$};
    \foreach \j in {0,...,\last}{
      \fill ({\r*cos(360*\j/\den)},{\r*sin(360*\j/\den)}) circle (1.05pt);
    }
  }

  \draw[moved] (0.72,0.18) -- (1.02,0.18);
  \draw[moved] (1.05,-0.18) -- (1.33,-0.18);
  \draw[moved] (1.36,0.38) -- (1.58,0.38);

  \node[font=\scriptsize, align=center] at (0,-2.92)
    {Dotted circles: old $T$ radii; solid circles: $S$ radii.};

\end{scope}

\node[align=center,font=\scriptsize,text width=2.8cm] at (0,1.15)
  {small normal perturbation \\$r_m\mapsto \rho_m$};

\node[align=center,font=\scriptsize,text width=11.7cm] at (0,-3.72)
  {Schematic, not to scale: for $Q=100$ the radii are all extremely close to the unit circle. 
  };
\end{tikzpicture}

\section{Acknowledgments}

We appreciate the helpful discussions with and the input by Andrei Caragea. It is important to acknowledge  that a ten-year-old open problem was resolved with a construction largely developed by an LLM, specifically ChatGPT. This is a testament to the remarkable potential of AI-assisted research in mathematics.

\end{document}